\documentclass{amsart}
\usepackage{amsmath, amsthm, amssymb, latexsym, graphics}

\newcommand{\fr}[1]{\mathfrak{#1}}
\newcommand{\ca}[1]{\mathcal{#1}}
\newcommand{\bk}{{{\boldsymbol{k}}}}
\newcommand{\bl}{{{\boldsymbol{l}}}}

\newcommand{\isom}{\simeq}
\newcommand{\ec}{\prec_{\exists}}
\newcommand{\Z}{\mathbb{Z}}

\newcommand{\cha}{\operatorname{char}}

\newcommand{\gi}{\Gamma_{\infty}}

\newtheorem{theorem}{Theorem}[section]
\newtheorem{lemma}[theorem]{Lemma}
\newtheorem{proposition}[theorem]{Proposition}

\newtheorem{definition}[theorem]{Definition}

\newtheorem{remark}[theorem]{Remark}

\theoremstyle{definition}

\title{Characterization of Extremal Valued Fields}

\author{Salih Azgin}
\address{Department of Mathematics and Statistics,
McMaster University, 1280 Main Street West,
Hamilton, L8S 4K1 Ontario, Canada}
\email{sazgin@math.mcmaster.ca}

\author{Franz-Viktor Kuhlmann}
\address{Department of Mathematics and Statistics,
University of Saskatchewan,
106 Wiggins Road,
Saskatoon, Saskatchewan, Canada S7N 5E6}
\email{fvk@math.usask.ca}

\author{Florian Pop}
\address{Department of Mathematics,
University of Pennsylvania,
209 South 33rd Street,
Philadelphia, PA 19104-6395, USA}
\email{pop@math.upenn.edu}

\begin{document}

\begin{abstract}
We characterize those valued fields for which the image of
the valuation ring under every polynomial in several variables contains
an element of maximal value, or zero.
\end{abstract}

\date{7.~12.~2009}

\thanks{
A major part of this research was done while the authors were
attending the o-minimality program at the Fields Institute, Jan.-June
2009. The authors would like to thank the Fields Institute for its
support and hospitality. The second author was partially supported by a
Canadian NSERC grant and by a sabbatical grant of the University of
Saskatchewan.}

\maketitle

\begin{section}{Introduction}
The notion of {\em extremality} for valued fields was introduced by
Yuri Ershov in~\cite{ershovext} in connection with valued skew fields
which are finite-dimensional over their center. It turns out that the
original definition given in that paper (and also in talks given by its
author) is flawed in the sense that there are no extremal valued fields
except algebraically closed valued fields, and Proposition 2 of that
paper is false. We fix this flaw by slightly modifying the definition of
extremality; see Definition~\ref{ext} below.

The notion of extremality, restricted to certain classes of polynomials,
has since become very useful for the characterization of various
properties of valued fields, cf.~\cite{kuhlmann}.

In valuation theory and particularly the model theory of valued fields,
power series fields and, more generally, maximal fields (valued fields
without proper immediate extensions) are usually known to have very good
properties. For instance, all of them are henselian, and what is more,
algebraically complete. So it seemed likely that all of them are also
extremal. Our results in this paper will show that this is not the case.

In the present paper, we obtain the following characterization of
extremal valued fields with residue characteristic $0$:

\begin{theorem}                                   \label{char0}
Extremal valued fields with residue characteristic $0$ are precisely
 \begin{itemize}
  \item[(i)] henselian valued fields whose value group is a
             $\mathbb{Z}$-group, and
  \item[(ii)] henselian valued fields whose value group is divisible
              and residue field is large.
 \end{itemize}
\end{theorem}

More generally, we prove in Section~\ref{sectchar}:

\begin{theorem}                                       \label{gen}
Let $\ca{K}=(K,\Gamma,\bk;v)$ be a valued field.
If $\ca{K}$ is extremal, then $\ca{K}$ is algebraically complete and
\begin{itemize}
  \item[(i)] $\Gamma$ is a $\mathbb{Z}$-group, or
  \item[(ii)] $\Gamma$ is divisible and $\bk$ is large.
 \end{itemize}
If $K$ is perfect with $\cha{\bk}=\cha{K}$, then also the
converse holds.
\end{theorem}

It remains an open question whether the above characterization also
holds in the case of mixed characteristic. In the case of non-perfect
valued fields, it does not hold. While for every field $k$, the formal
Laurent Series Field $k((t))$ with its canonical valuation $v_t$ is
extremal according to Theorem~\ref{zgroup} below, we will show in
Section~\ref{sectchar}:

\begin{theorem}                             \label{counter}
There exist non-extremal algebraically complete valued fields of equal
positive characteristic with value group a $\Z$-group, and such that
under a coarsening of their valuation, they are still not extremal, have
divisible value group and non-perfect large residue field.
\end{theorem}

\medskip
See the following section for the definitions of {\em algebraically
complete valued field}, {\em $\mathbb{Z}$-group} and {\em large field}.

\medskip
The authors would like to thank Sergei Starchenko for poining out the
flaw in the definition of extremality and for providing the first
example given in Remark~\ref{ershovw} below.
%

\end{section}

\begin{section}{Preliminaries}

We assume familiarity with the basic concepts of valued fields
and their model theoretic properties. We consider valued fields as
three-sorted structures
$$\ca{K}=(K,\Gamma,\bk; v)$$
where $K$ is the underlying field, $\Gamma$ is the value group,
$\bk$ is the residue field,
$v: K^\times \to \Gamma$ is the valuation,
with valuation ring
$$\ca{O}_v:=\{a \in K:\ v(a) \geq 0\}$$ and maximal ideal
$\fr{m}_v:=\{a \in K: v(a)>0\}$ of $\ca{O}$. We have the
residue class map
\begin{align*}
     \bar{}&:\ca{O}_v \to \bk  \\
           &a \mapsto a + \fr{m}_v
  \end{align*}
and the {\em residue characteristic of $K$} is the characteristic of $\bk$.
We often refer to $K$ as the valued field, instead of $\ca{K}$. When there are
more than one valuations defined on the same field we use the notations
$Kv$ and $vK$ to denote the residue field and value group of $K$ respectively
for the valuation $v$. We also use the notation $av$ to denote the residue class
of an element in $\ca{O}_v$.
For a subset $\Delta$ of an ordered group $\Gamma$ we use the notation
$$\gamma < \Delta$$
as a shorthand for $\gamma < \delta$ for all $\delta \in \Delta$. We set
$v(0):=\infty > \Gamma$, $\gi:=\Gamma \cup \{\infty\}$, $-\infty<\Gamma$
and for $\gamma \in \Gamma$ the intervals $(-\infty,\gamma)$ and
$(\gamma, \infty)$ are defined as usual. We use $m,n, \dots$ to denote
elements of $\mathbb{N}$ unless specified otherwise.

\begin{definition}
A valued field $\ca{K}$ is called {\em algebraically
maximal} if it does not admit proper immediate algebraic extensions
(that is, extensions which preserve value group and residue field).
Since henselizations are immediate algebraic extensions, every
algebraically maximal field is henselian.
\end{definition}

\begin{definition}
A valued field $\ca{K}$ is called {\em algebraically
complete} if it is henselian and for every finite extension
$(L,\Delta, \bl; w)$ we have
\begin{equation}                            \label{dl}
[L:K]=(\Delta:\Gamma)[\bl:\bk]\>.
\end{equation}
If equation (\ref{dl}) holds, then the extension is called
{\em defectless}.
\end{definition}

\begin{definition}
An ordered group is {\em regular} if for each $n$ every open interval
that contains $n$ elements contains an $n$-divisible element. A {\em
$\mathbb{Z}$-group} is a regular group which is discrete, i.e., has a
smallest positive element. An ordered group is {\em dense} if for any
two elements $\alpha<\beta$ there is an element $\gamma$ in the group
such that $\alpha<\gamma<\beta$.
\end{definition}

\medskip\noindent
Note that $\mathbb{Z}$-groups are exactly the ordered groups which
are elementarily equivalent to $\mathbb{Z}$. If $\Gamma$ is dense and
regular then for each $n$ and $\gamma \in \Gamma$
there is an increasing sequence $\{n\gamma_\rho\}_{\rho<\lambda}$ which is
cofinal in $(-\infty,\gamma)$.
We state two useful facts on regular groups, from
\cite{robinson-zakon} and \cite{conrad} respectively.

\begin{theorem}\label{reg1} A regular group is either a $\mathbb{Z}$-group
or it is dense.
\end{theorem}

\begin{theorem}\label{reg2} An ordered group $\Gamma$ is regular if and only if
$\Gamma/\Delta$ is divisible for every nonzero convex subgroup
$\Delta$ of $\Gamma$.
\end{theorem}

\begin{definition}
Take $\mathcal{A}\subseteq \mathcal{B}$ to be an extension of two
structures of a first order language. Then $\mathcal{A}$ is {\em
existentially closed in} $\mathcal{B}$ if every existential sentence
with parameters from $\mathcal{A}$ that holds in $\mathcal{B}$ also
holds in $\mathcal{A}$.
\end{definition}

\begin{definition}
A field $\bk$ is {\em large} if every curve over $k$ which has a
smooth $k$-rational point, has infinitely many such points.
\end{definition}

\noindent
Algebraically closed fields, real closed fields, pseudo-algebraically
closed fields, fields equipped with a henselian valuation are all large
fields. Finite fields are not large. Note that a field $\bk$ is large if
and only if it is existentially closed in $\bk((t))$ (\cite{pop}; see
also \cite{kuhlmann3}).


For the following result, see Theorem 17 of \cite{kuhlmann3}.
\begin{theorem}                             \label{lec}
Suppose that $\bk$ is a large and perfect subfield of a field $F$. If
$\bk$ is the residue field of a valuation on $F$ which is trivial on
$\bk$, then $\bk$ is existentially closed in $F$ (as a field).
\end{theorem}

Finally, we will need the following two well known technical lemmas.
\begin{lemma}                               \label{ww}
Let $\ca{K}=(K,\Gamma,\bk;v)$ be a valued field such that $v=w\circ
\bar w$. Then $\ca{K}$ is henselian if and only if $w$ is henselian on
$K$ and $\bar w$ is henselian on the residue field $Kw$ of $K$ under
$w$. The same holds for ``algebraically complete'' in the place of
``henselian''.
\end{lemma}

By use of Hensel's Lemma, one proves:
\begin{lemma}                               \label{emb}
Let $\ca{K}=(K,\Gamma,\bk;v)$ be a perfect and henselian valued field.
Then every embedding of a subfield of $\bk$ such that $v$ is trivial on
the image can be extended to an embedding of $\bk$ such that $v$ is
trivial on the image.
\end{lemma}

\end{section}

\begin{section}{Properties of extremal fields}      \label{sectprop}

\begin{definition}                           \label{ext}
A valued field $\ca{K}$ is {\em extremal} if for every
multi-variable polynomial $F(X_1, \dots , X_n)$ over $K$ the set
$$\{v(F(a_1,\dots,a_n)): a_1,\dots a_n \in \ca{O}_v\}
\subseteq \gi$$ has a maximal element.
\end{definition}

\begin{remark}                                \label{ershovw}
\rm The original definition presented in \cite{ershovext} asks for a
maximal element of the set $$\{v(F(a_1,\dots,a_n)): a_1,\dots a_n\in K\}
\subseteq \gi$$ where $K$ and $F$ are as above. This condition is not
satisfied for the polynomial $$F(X,Y)=X^2+(XY-1)^2$$ over the Laurent
series field $\mathbb{R}((t))$ as $$v\big(F(t^n, t^{-n})\big)=2n$$
and $F(a,b) \neq 0$ for all $a,b \in \mathbb{R}((t))$. Hence,
Proposition 2 of~\cite{ershovext} does not hold for the
original definition.

Suppose that $K$ is a valued field which is not algebraically closed.
Take a polynomial $f(x)=x^n+a_{n-1}x^{n-1}+\cdots+a_0 \in K[x]$ with no
zeros in $K$ and let
$$G(X,Y)=X^n+a_{n-1}X^{n-1}Y+\cdots+a_0Y^n$$
be its homogenization. Then the polynomial $F(X,Y)=G(X,XY-1)$ does not
satisfy the above condition. Indeed, $G(X,Y)$ can only be zero if $Y$
is zero and then, consequently, also $X$ is zero. So $F(X,Y)$ can only
be zero if $XY-1$ and $X$ are zero, which is impossible. On the other
hand, for $a\ne 0$ we have that $vF(a,a^{-1})=v(a^n)=nv(a)$, which shows
that $\{v(F(a,b)\mid a,b\in K\}$ has no maximum.
\end{remark}

The following is a consequence of Theorem~1.5 of \cite{kuhlmann}:
\begin{proposition}                         \label{eam}
Every extremal field is algebraically maximal and hence henselian.
\end{proposition}

%
%

\begin{proposition}                                  \label{exreg}
Let $\ca{K}=(K,\Gamma, \bk; v)$ be a valued field with $\Gamma$ dense
and regular. If $\ca{K}$ is extremal, then $\Gamma$ is divisible.
\end{proposition}

\begin{proof}
Let $\gamma \in \Gamma$ be positive and not divisible by $n>1$.
Consider the polynomial
$$F(x,y)=x^{4n}+\epsilon(xy-\epsilon^2)^n+\epsilon^2y^{4n}$$
over $K$ where $v(\epsilon)=\gamma$. Note that for all
$a, b \in K^{\times}$,
$$v(a^{4n}),\quad v(\epsilon(ab-\epsilon^2)^n),\quad v(\epsilon^2b^{4n})$$
are distinct elements of $\gi$ since $\gamma$ is
not $n$-divisible. Hence the valuation of $F(a,b)$ is equal to
the minimum of the three distinct elements above.

We claim that for all $a, b \in \ca{O}_v$,
$v(F(a,b))<4n\gamma+\gamma$. Assume otherwise.
Then $v(a^{4n}) \geq 4n\gamma+\gamma$ and since
equality is not possible,
$$4nv(a)>4n\gamma+\gamma.$$
Also, $v(\epsilon(ab-\epsilon^2)^n) \geq 4n\gamma+\gamma$ and so
$v(a)+v(b) = 2\gamma$. Consequently,
$$4nv(b)=8n\gamma-4nv(a)$$
which gives $4nv(b)< 4n\gamma -\gamma$. Then
$$v(F(a,b))\leq v(\epsilon^2b^{4n})=2\gamma+4nv(b)< 4n\gamma + \gamma,$$
contradiction.

Since $\Gamma$ is dense and regular we can take a sequence
$\{4n\delta_{\rho}\}_{\rho<\lambda}$ in $\Gamma$ which
is cofinal in the interval
$(4n\gamma,4n\gamma+\gamma)$. For each $\rho<\lambda$, pick
$a_{\rho} \in \ca{O}_v$ with $v(a_{\rho})=\delta_{\rho}<2\gamma$ and let
$b_\rho:=\epsilon^2a_\rho^{-1} \in \ca{O}_v$. Then for all $\rho$
$$v(a_\rho^{4n})=4n\delta_\rho <4n\gamma+\gamma<
v(\epsilon^2b_\rho^{4n})=2\gamma+4n(2\gamma - \delta_\rho)$$
and hence $v(F(a_\rho,b_\rho))=4n\delta_\rho$.
Together with the previous claim, this shows that
$\{v(F(a,b)):a,b \in \ca{O}_v\}$ has no maximal element and so $\ca{K}$
is not extremal.
\end{proof}

\begin{proposition}                                   \label{exnreg}
Let $\ca{K}=(K,\Gamma, \bk, v)$ be an extremal valued field. Then for
each nonzero convex subgroup $\Delta$ of $\Gamma$, the quotient
group $\Gamma/\Delta$ is divisible.
\end{proposition}

\begin{proof}
Let $\gamma \in \Gamma$ be such that the coset $\gamma + \Delta$ is not
divisible by $n$ in $\Gamma/\Delta$. Consider the polynomial (as
introduced in the previous theorem)
$$F(x,y)=x^{4n}+\epsilon(xy-\epsilon^2)^n+\epsilon^2y^{4n}$$
over $K$ where $v(\epsilon)=\gamma$.

The arguments in the proof of Proposition~\ref{exreg} can be applied to
the ordered group $\Gamma / \Delta$ to conclude that for all $a, b \in
\ca{O}_v$,
\begin{equation}                            \label{Delta1}
v(F(a,b))<4n\gamma+\gamma+\Delta\>.
\end{equation}

\medskip\noindent
{\bf Claim:} $\{v(F(a,b):a,b \in \ca{O}_v\}$ has no maximal element.
Take $a,b \in \ca{O}_v$ and set $\theta=v(F(a,b))$. We aim to find
$a',b' \in \ca{O}_v$ with $v(F(a',b'))>\theta$. Take any positive
$\delta\in\Delta$. From (\ref{Delta1}) it follows that
\begin{equation}                            \label{Delta2}
\theta +4n\delta < 4n\gamma+\gamma-4n\delta\>.
\end{equation}

Note that $v(F(\epsilon,\epsilon))=4n\gamma$, so we assume that
$4n\gamma<\theta$. As $v(\epsilon)=\gamma$ is not divisible by $n$,
$\theta$ is equal to the minimum of $4nv(a)$, $\gamma+nv(ab-\epsilon^2)$
and $2\gamma+4nv(b)$. Therefore,
$$4n\gamma < \theta \leq v(\epsilon(ab-\epsilon^2)^n)$$
and hence
$$v(b)=2\gamma - v(a)\>.$$

\smallskip\noindent
Assume that $\theta=4nv(a)$. Together with $4nv(a) < 4n\gamma+\gamma$,
(\ref{Delta2}) shows that
\begin{equation}                            \label{1}
4nv(a)+4n\delta< \gamma + 4n\gamma -4n\delta
< 2\gamma + 8n\gamma - 4nv(a)-4n\delta\>.
\end{equation}
Now take $a' \in \ca{O}_v$ with $v(a')=v(a)+\delta$ and let
$b'=\epsilon^2/a' \in \ca{O}_v$. Then $v(b')=2\gamma-v(a)-\delta$ and
(\ref{1}) implies that $4nv(a')<2\gamma+4nv(b')$. Therefore
$v(F(a',b'))=4nv(a')=4nv(a)+4n\delta>\theta$ as required.

\medskip\noindent
Next assume that $\theta=\gamma+nv(ab-\epsilon^2) \geq 4n\gamma$. Then
$\theta<4nv(a)$ and $\theta<2\gamma+4nv(b)$. With $b'=\epsilon^2/a \in
\ca{O}_v$, we have $v(b')=v(b)$ and therefore $v(F(a,b')) \geq \min
\{4nv(a), 2\gamma+4nv(b')\} > \theta$.

\medskip\noindent
It remains to consider the case $\theta=2\gamma+4nv(b)$.
Together with $2\gamma+4nv(b)=\theta< 4n\gamma+\gamma$,
(\ref{Delta2}) shows that
\begin{eqnarray*}
2\gamma+4nv(b)+4n\delta & < & 4n\gamma+\gamma-4n\delta +
(4n\gamma+\gamma - 2\gamma-4nv(b))-4n\delta\\
 & = & 8n\gamma - 4nv(b) - 4n\delta\>.
\end{eqnarray*}
Now take $b' \in \ca{O}_v$ with $v(b')=v(b)+\delta$ and let
$a'=\epsilon^2/b' \in \ca{O}_v$. Then $v(a')=2\gamma - v(b) - \delta$
and the above inquality implies that $2\gamma + 4nv(b')<4nv(a')$.
Therefore
$v(F(a',b'))=2\gamma+4nv(b)+4n\delta>\theta$
as required. This completes the proof of our claim and we conclude that
$\ca{K}$ is not extremal.
\end{proof}

\begin{theorem}                             \label{ndnzne}
Let $\ca{K}=(K,\Gamma,\bk;v)$ be an extremal valued field. Then
$\Gamma$ is divisible or a $\mathbb{Z}$-group.
\end{theorem}

\begin{proof}
Proposition~\ref{exnreg} shows that $\Gamma$ has no nonzero convex
subgroup $\Delta$ such that $\Gamma/\Delta$ is not divisible. Therefore,
$\Gamma$ is regular by Theorem~\ref{reg2}. If $\Gamma$ is not a
$\mathbb{Z}$-group then, by Theorem~\ref{reg1}, $\Gamma$ is dense and
Proposition~\ref{exreg} yields that $\Gamma$ is divisible.
\end{proof}

\begin{proposition}                                   \label{exnlar}
Let $\ca{K}=(K,\Gamma,\bk;v)$ be an extremal valued field with
non-algebraically closed residue field $\bk$ and divisible value group
$\Gamma$. Then there is no polynomial $f(Y,Z)$ over $\ca{O}_v$ such
that
 \begin{enumerate}
  \item[(i)] there are $\alpha_1, \dots, \alpha_m
  \in \bk$ such that the equation $$\bar{f}(\alpha, z)=0$$ over $\bk$
  has no solution in $\bk$ whenever $\alpha \neq \alpha_1,\dots,\alpha_m$;
  \item[(ii)] for all $\epsilon \in \fr{m}_v$,
  $f(\epsilon, Z)=0$ has a solution in $\ca{O}_v$.
 \end{enumerate}
\end{proposition}

\begin{proof}
Let $\beta_0,\dots,\beta_{n-1} \in \bk$ be such that
$x^n+\cdots + \beta_1 x+\beta_0=0$ has no solutions in $\bk$.
Consider the polynomial
$$F(X,Y)= X^n + \cdots + b_1 X Y^{n-1} + b_0 Y^n$$
over $\ca{O}_v$ where $\bar{b}_i=\beta_i$ and $b_i=0$ if $\beta_i=0$,
for $i=0,\dots,n-1$. Note that for all positive $\gamma \in \Gamma$,
$$v(F(a,b)) \geq \gamma \Longrightarrow v(a),v(b) \geq \gamma/n.$$
Let $G(X,Y,Z):=F(F(X,Y),Z)$. For positive $\gamma \in \Gamma$, if
$v(G(a,b,c)) \geq \gamma$ then $v(a), v(b) \geq \gamma/n^2$ and
$v(c) \geq \gamma/n$.

Let $a_1,\dots,a_m \in \ca{O}_v$ be such that $\bar{a}_j=\alpha_j$ and
$a_j=0$ if $\alpha_j=0$ (where $\alpha_j$ is as in (i) above) for
$j=1,\dots,m$. Consider the polynomial
$$H(X,Y,Z):=G\big(X, f(Y,Z), XY(Y-a_1)\cdots (Y-a_m)-\epsilon\big)$$
over $K$ where $v(\epsilon)=\gamma>0$. We claim that
$$v(H(a, b, c))<n^2\gamma$$
for all $a,b,c \in \ca{O}_v$. Otherwise, we would have
$$v(a) \geq \gamma \text{  and  } v\big(ab(b-a_1)\cdots
(b-a_m)-\epsilon\big) \geq n\gamma.$$
Therefore $v(ab(b-a_1)\cdots(b-a_m))=v(\epsilon)=\gamma$, which in
turn gives
$$v(a)=\gamma, \quad v(b)=0 \text{  and  } v(b-a_i)=0 \quad \text{for}
\quad i=1,\dots,m.$$
So $\bar{b} \neq \alpha_1, \dots, \alpha_m$ and
hence for all $z \in \ca{O}_v$, $\bar{f}(\bar{b},\bar{z})) \neq 0$.
Then $v(H(a,b,c))=v(f(b,c))=0$, contradiction, and we establish the claim.

\medskip\noindent
Let $\{\gamma_\rho\}_{\rho<\lambda}$ be a decreasing sequence such that
$\gamma_\rho \rightarrow 0$ and pick $b_\rho \in \fr{m}_v$ with
$v(b_\rho)=\gamma_\rho$ for all $\rho<\lambda$. Let $a_\rho$ be such
that $a_{\rho}b_\rho(b_\rho-a_1)\dots(b_\rho-a_m)=\epsilon$. Note that
$v(a_\rho) \rightarrow \gamma$, and in particular we may assume that
$a_\rho \in \ca{O}_v$ for all $\rho<\lambda$. Pick $c_\rho \in \ca{O}_v$
such that $f(b_\rho, c_\rho)=0$ for each $\rho<\lambda$. Then
$$v\big(H(a_\rho, b_\rho, c_\rho)\big)=n^2v(a_\rho)$$
and hence $\{v(H(a_\rho, b_\rho, c_\rho))\}$
is cofinal in $(0,n^2\gamma)$. We conclude that $\ca{K}$ is not extremal.
\end{proof}

\begin{theorem}                             \label{dnlne}
Suppose that $\ca{K}=(K,\Gamma,\bk; v)$ has a divisible value group.
If $\ca{K}$ is extremal, then $\bk$ is large.
\end{theorem}

\begin{proof} Since $\bk$ is not large, there is a curve $C$ defined
over $\bk$ which has a smooth $\bk$-rational point, but has only
finitely many $\bk$-rational points. But then by the theory of algebraic
curves, the curve $C$ is birational to a curve in the affine $\bk$-plane
$C_h=V(h)\subset{\bf A}^2_{\bk}$, where $h=h(Y,Z)\in\bk[Y,Z]$ is a
polynomial in the variables $Y,Z$ satisfying the following:

\begin{enumerate}
\item[(i)$'$] $C_h(\bk)=\{(\alpha,\beta)\in\bk^2\mid h(\alpha,\beta)=0\}$
is finite.
\smallskip\noindent
\item[(ii)$'$] $h(0,0)=0$ and $\partial h/\partial Z(0,0)\neq0$, hence
in particular, $(0,0)$ is a smooth $k$-rational point of $C_h$. (Actually
these conditions imply that $h(Y,Z)$ has the form
$h(Y,Z)=\gamma_{10}Y+\gamma_{01}Z+(\hbox{non-linear terms})$
with $\gamma_{01}\neq0$.)
\end{enumerate}

\smallskip\noindent
Setting $h(Y,Z)=\sum_{i,j}\gamma_{ij}Y^iZ^j$, let
$f(Y,Z)=\sum_{i,j}c_{ij}Y^iZ^j\in\ca{O}[Y,Z]$ be a preimage
of $h(Y,Z)$ such that $c_{ij}=0$ if $\gamma_{ij}=0$, hence
$c_{ij}\in\ca{O}^\times$ if $\gamma_{ij}\neq 0$. Finally, let
$\{\alpha_1=0,\dots,\alpha_m\}$ be the set of all the $Y$-coordinates
of the points $(\alpha,\beta)\in C_h(\bk)$. We claim that
$f(Y,Z)\in\ca{O}[Y,Z]$ satisfies the conditions~(i),~(ii)
from Proposition~\ref{exnlar}. Indeed, condition~(i) is obviously
satisfied, because $\overline f(Y,Z)=h(Y,Z)$, and $h(Y,Z)$
satisfies condition~(i)$'$. For condition~(ii), proceed as
follows: For $\epsilon\in\fr{m}_v$ set
$f_\epsilon(Z):=f(\epsilon,Z)\in\ca{O}[Z]$. Then
$\overline f_\epsilon(Z)=h(0,Z)$, hence $Z=0$ is a simple
zero of $\overline f_\epsilon(Z)=h(0,Z)$, by condition~(ii)$'$.
Therefore, by Hensel's Lemma there exists a unique
$\eta\in \fr{m}_v$ such that $f_\epsilon(\eta)=0$, hence
equivalently, $f(\epsilon,\eta)=f_\epsilon(\eta)=0$.

\smallskip\noindent
By Proposition~\ref{exnlar} we conclude that $\ca{K}$ is not extremal.
\end{proof}

The results proved so far enable us to give a version of Proposition~1
of \cite{ershovext} for the modified notion of ``extremal field''.

\begin{definition}
We say that a basis $b_1,\ldots,b_n$ of a valued field extension
$(K,\Gamma, \bk; v)\subset (L,\Delta, \bl; w)$ is a {\em valuation
basis} if for all choices of $c_1,\ldots,c_n\in K$,
\[w\sum_{i=1}^{n}c_ib_i\;=\;\min_i wc_ib_i\>.\]
Note that every finite defectless extension admits a valuation basis.
\end{definition}

\begin{lemma}
Every finite defectless extension of an extremal field is again an
extremal field.
\end{lemma}
\begin{proof}
Take an extremal field $\ca{K}=(K,\Gamma, \bk; v)$ and a finite defectless
extension $(L,\Delta, \bl; w)$ of degree $m$; we wish to show that the
latter is an extremal field. From Theorem~\ref{ndnzne} we know that
$\Gamma$ is divisible or a $\mathbb{Z}$-group. Therefore, all cosets of
$\Gamma$ in $\Delta$ admit representatives that are either $0$ or lie
between $0$ and the smallest positive element of $\Gamma$. Consequently,
we can choose a valuation basis $b_1,\ldots,b_m$ of the extension such
that the values $wb_i$ have the same property. Write
$\overline{Y}=(Y_1,\ldots,Y_m)$ and take
\[h(\overline{Y})\>=\>N_{L(\overline{Y})|K(\overline{Y})}
\left(\sum_{i=1}^{m} b_i Y_i\right)\]
be the norm form with respect to the basis $b_1,\ldots,b_m$ of the
extension $L(\overline{Y})|K(\overline{Y})$.

Take a polynomial $F(X_1, \dots , X_n)$ over $L$; we wish to show that
the set
\[\{w(F(a_1,\dots,a_n)): a_1,\dots a_n \in \ca{O}_w\}\]
has a maximal element. Denote by $G(\overline{Z})$ the
polynomial obtained from $F$ by substituting $\sum_{i=1}^{m} b_i Z_{ij}$
for $X_i$, $1\leq i\leq n$. The polynomial $G(\overline{Z })$
can be written as $G(\overline{Z})=\sum_{i=1}^{m} b_i
G_i (\overline{Z})$ with $G_i (\overline{Z})\in
K[\overline{Z}]$ for every $i$. Now let
\[H(\overline{Z})\>=\>h(G_1 (\overline{Z}),\ldots,
G_m (\overline{Z}))\>=\>N_{L(\overline{Y})|K(\overline{Y})}
(G(\overline{Z}))\in K[\overline{Z}]\;\>.\]
Since $\ca{K}$ is extremal, there exist elements $c_{ij}\in\ca{O}_v$,
$1\leq i\leq m$, $1\leq j\leq n$, such that
\[vH(\overline{c})\>=\>\max\{wH(\overline{c'})\mid c'_{ij}\in
\ca{O}_v\}\>.\]
For $1\leq j\leq n$, we set $d_j=\sum_{i=1}^{m} b_i c_{ij}$. We wish to
show that
\[wF(\overline{d})\>=\>\max\{wF(\overline{d'})\mid d'_j\in
\ca{O}_w\}\>.\]
Note that $vN_{L|K}(a)=mwa$ for every $a\in L$ since $\ca{K}$ is
henselian by Proposition~\ref{eam}. Take $d'_j\in\ca{O}_w$ and write
$d'_j=\sum_{i=1}^{m} b_i c'_{ij}$ with $c'_{ij}\in K$. Since the $b_i$
form a valuation basis,
\[0\>\leq\>vd'_j\>=\>w\sum_{i=1}^{m} b_ic'_{ij}\>=\>
\min_i wb_ic'_{ij}\>.\]
Hence for $1\leq i\leq m$ and $1\leq j\leq n$, $wb_i+vc'_{ij}\geq 0$. By
our assumptions on the values $wb_i$, this implies that $c'_{ij}\in
\ca{O}_v$. Now we compute
\begin{eqnarray*}
wF(\overline{d'}) & = & \frac{1}{m}\,vN_{L|K}(F(\overline{d'}))
\>=\>\frac{1}{m}\,vH(\overline{c'})\\
 & \leq & \frac{1}{m}\,vH(\overline{c}) \>=\>
\frac{1}{m}\,vN_{L|K}(F(\overline{d})) \>=\> wF(\overline{d})\>.
\end{eqnarray*}

\end{proof}

This lemma yields:
\begin{theorem}                             \label{eac}
a) \ Every extremal valued field is algebraically complete.
\par\smallskip\noindent
b) \ Every finite extension of an extremal field is again an extremal
field.
\end{theorem}
\begin{proof}
a): \ By Proposition~\ref{eam}, every extremal field is algebraically
maximal. Hence our assertion follows from the preceding lemma once we
know that a valued field is algebraically complete if each of its finite
defectless extensions is algebraically maximal. This holds by
Corollary~2.10 of \cite{kuhlmann}. For the convenience of the reader, we
give a sketch of the proof. Every finite extension of an algebraically
complete field is again algebraically complete, hence algebraically
maximal. For the converse, take a valued field $\ca{K}$ such that every
finite extension is algebraically maximal. Then $\ca{K}$ is itself
algebraically maximal, hence henselian. Take a finite extension $\ca{L}$
of $\ca{K}$. In order to show that it is defectless, it suffices to show
that its normal hull $\ca{N}$ is a defectless extension of $\ca{K}$.
Denote by $\ca{N}'$ the maximal separable subextension of
$\ca{N}|\ca{K}$, and by $\ca{R}$ the ramification field of this Galois
extension. Then $\ca{R}|\ca{K}$ is defectless. On the other hand,
$\ca{N}'|\ca{K}$ is a $p$-extension, where $p$ is the characteristic of
the residue field. Consequently, $\ca{N}|\ca{K}$ is a tower of
(separable or purely inseparable) extensions of degree $p$. As
$\ca{R}|\ca{K}$ is defectless, there is a maximal field $\ca{L}'$ in the
tower such that $\ca{L}'|\ca{K}$ is defectless. By assumption, $\ca{L}'$
is algebraically maximal. If $\ca{L}'\ne \ca{N}$, then the next larger
field $\ca{L}''$ in the tower is an extension of degree $p$ of $\ca{L}'$
and it is not immediate, hence defectless. By multiplicativity of the
defect, we find that $\ca{L}''|\ca{K}$ is defectless, contradicting the
maximality of $\ca{L}'$.

\par\smallskip\noindent
b): \ By part a), every finite extension of an extremal field is
defectless. Hence the assertion of part b) follows from the preceding
lemma.
\end{proof}
\end{section}

\begin{section}{Characterization of extremal fields} \label{sectchar}

The first non-algebraically closed examples of extremal fields were
provided by Proposition 2 in \cite{ershovext}:

\begin{theorem}                                   \label{zgroup}
Let $\ca{K}=(K,\Gamma, \bk; v)$ be an algebraically complete
valued field with $\Gamma\isom\mathbb{Z}$. Then $\ca{K}$ is extremal.
\end{theorem}
By Remark~\ref{ershovw}, this theorem would be false with the
definition of extremality as stated in \cite{ershovext}. However, it is
easy to check that the proof of Theorem~\ref{zgroup} given in
\cite{ershovext} is valid with the revised definition of extremality.

Note that extremality is a first order condition for valued fields.
Hence we can apply the following Ax-Kochen-Ershov Principle for tame
valued fields proved in \cite{kuhlmann2}, to show extremality for larger
classes of valued fields:

\begin{theorem}                             \label{AKE}
Let $\ca{K}=(K,\Gamma, \bk; v)$ and $\ca{L}=(L,\Delta, \bl; w)$ be two
perfect algebraically complete valued fields with $\cha{\bk}=\cha{K}$
and $\cha{\bl}=\cha{L}$.
If $\bk\equiv\bl$ and $\Gamma\equiv\Delta$, then $\ca{K}\equiv\ca{L}$.
If in addition $\ca{L}$ is an extension of $\ca{K}$, and if
$\bk\prec\bl$ and $\Gamma\prec\Delta$, then $\ca{K}\prec\ca{L}$. The
same holds with ``$\ec$'' in the place of ``$\prec$'' as soon as $\ca{K}$
is perfect and algebraically complete with $\cha{\bk}=\cha{K}$.
\end{theorem}
Here, ``$\equiv$'' denotes ``elementarily equivalent, ``$\prec$''
denotes ``elementary extension'' and ``$\ec$'' denotes ``existentially
closed in''.

If the residue field has characteristic 0, then ``perfect algebraically
complete'' is equivalent to ``henselian''; this is a consequence of the
Lemma of Ostrowski. As a corollary to Theorems~\ref{zgroup}
and~\ref{AKE}, we obtain:

\begin{theorem}                             \label{zge}
Let $\ca{K}=(K, \Gamma, \bk; v)$ be a henselian valued field with
$\Gamma$ a $\mathbb{Z}$-group and $\cha{\bk}= \cha{K}=0$. Then $\ca{K}$
is extremal.
\end{theorem}

The following theorem includes also the case of fields of positive
characteristic:
\begin{theorem}                               \label{divisible}
Let $\ca{K}=(K, \Gamma, \bk;v)$ be a perfect algebraically complete
valued field with divisible value group $\Gamma$ and large residue field
$\bk$ with $\cha{\bk}=\cha{K}$. Then $\ca{K}$ is extremal.
\end{theorem}
\begin{proof}
In view of Theorem~\ref{AKE}, we only have to prove our theorem in the
case where $\Gamma$ is the ordered additive group of the real numbers.
Take any polynomial $F\in K[X_1,\ldots,X_n]$. Take $\ca{K}^*= (K^*,
\Gamma^*, \bk^*;v^*)$ to be a $|K|^+$-saturated elementary extension of
$\ca{K}$. Then $\Gamma^*$ is an elementary extension of $\Gamma$ and
$\bk^*$ is an elementary extension of $\bk$. Now we distinguish two
cases.

\par\smallskip\noindent
\underline{Case 1}: \ $\{v(F(a_1,\dots,a_n)): a_1,\dots a_n \in
\ca{O}_v\} \setminus \{\infty\}$ is cofinal in $\Gamma$. We wish to show
that $F$ has a zero in $\ca{O}_v^n$.

By our choice of $\ca{K}^*$ there are $a_1^*,\dots,a_n^*\in
\ca{O}_{v^*}$ such that $v^*(F(a_1^*,\dots,a_n^*))>\Gamma$.
Hence there exists a valuation $w$ on $\ca{K}^*$ that is coarser than
$v^*$, satisfies $wF(a_1^*,\dots,a_n^*)>0$, but is trivial on $K$. So we
can consider $K$ as a subfield of the residue field $K^*w$ of $K^*$
under $w$. We write $v^*=w\circ\bar{w}$ where $\bar{w}$ is the valuation
induced by $v^*$ on $K^*w$. Note that $a_1^*, \dots,a_n^*\in\ca{O}_{w}$.
We obtain that $0=F(a_1^*,\dots, a_n^*)w= F(a_1^*w,\dots,a_n^*w)$ and
that $a_1^*w,\dots,a_n^*w\in \ca{O}_{\bar{w}}$. Denote the value group
of $\bar{w}$ on $K^*w$ by $\Gamma'$. Since $\Gamma$ is
divisible, $\Gamma\ec\Gamma'$. Further, $\bk^*$ is equal to the residue
field of $K^*w$ under $\bar w$, and $\bk\ec\bk^*$. Hence by
Theorem~\ref{AKE}, $\ca{K}\ec (K^*w,\Gamma',\bk^*;\bar w)$.
Since $F$ has a zero in the valuation ring of the latter field, this
implies that there are $a_1,\dots,a_n\in \ca{O}_v$ such that
$F(a_1,\dots,a_n) =0$.

\par\smallskip\noindent
\underline{Case 2}: \ $\{v(F(a_1,\dots,a_n)): a_1,\dots a_n \in
\ca{O}_v\}$ is not cofinal in $\Gamma=\mathbb{R}$. Then there is some
real number $r$ which is the supremum of this set. We wish to show that
it is a member of the set.

By our choice of $\ca{K}^*$ there are $a_1^*,\dots,a_n^*\in
\ca{O}_{v^*}$ such that $\delta^*:=v^*(F(a_1^*,\dots,a_n^*))\geq
v(F(a_1,\dots,a_n))$ for all $a_1,\dots,a_n\in \ca{O}_v$. On the other
hand, $\delta^*\leq r$ in $\Gamma^*$. So $0\leq r-\delta^*<s$ for all
positive $s\in \mathbb{R}=\Gamma$. We take $\Gamma_0$ to be the largest
convex
subgroup of $\Gamma^*$ such that $\Gamma_0\cap\Gamma=\{0\}$. Then we
take $w$ to be the coarsening of $v^*$ with respect to $\Gamma_0$, that
is, $\ca{O}_{w}$ is generated over $\ca{O}_{v}$ by all elements whose
values under $v$ lie in $\Gamma_0$, and the value group of $w$ on $K^*$
is $\Gamma^*/\Gamma_0$. Again, we write $v^*=w\circ\bar{w}$ where
$\bar{w}$ is the valuation induced by $v^*$ on the residue field $K^*w$.
We observe that $v=w$ on $K$ (after identification of equivalent
valuations). Further, $w(F(a_1^*,\dots,a_n^*))=r$.

Now we have that $(K^*,\Gamma^*/\Gamma_0,K^*w; w)$ is an extension of
$\ca{K}$. We have canonical embeddings of $\bk$ in $K^*w$ and of
$\Gamma$ in $\Gamma^*/\Gamma_0$. Since $\Gamma$ is divisible, $\Gamma\ec
\Gamma^*/\Gamma_0$. The residue field $K^*w$ has a valuation $\bar w$
with residue field $\bk^*$. Being an elementary extension of $K$, also
$K^*$ is perfect. It follows that its residue fields $K^*w$ and
$\bk^*$ are perfect. Since every algebraically complete valued field is
henselian, Lemma~\ref{ww} shows that $K^*w$ is henselian under $\bar w$.
By Lemma~\ref{emb}, the canonical embedding of $\bk$ in $K^*w$ can be
extended to an embedding of $\bk^*$ in $K^*w$. Via this embedding we may
assume that $\bk^*$ is a subfield of $K^*w$. Since $\bk^*$ is an
elementary extension of $\bk$, it is again a large field. Thus by
Theorem~\ref{lec}, $\bk^*\ec K^*w$ and consequently, $\bk\ec K^*w$. It
now follows from Theorem~\ref{AKE} that $\ca{K}\ec
(K^*,\Gamma^*/\Gamma_0, K^*w; w)$. Hence there are $a_1,\dots,a_n\in
\ca{O}_v$ such that $vF(a_1,\dots,a_n) =r$.
\end{proof}

\par\smallskip
Theorem~\ref{gen} now follows from Theorems~\ref{zge},
\ref{divisible}, \ref{ndnzne}, \ref{dnlne} and~\ref{eac}.

\par\bigskip
We will now give an example that will prove
Theorem~\ref{counter}. Its construction is given in the paper
\cite{kuhlmann4}; we will use the same notation as in that paper. Taking
the valued field $(K,v)$ appearing in the construction to be
$(\mathbb{F}_p((t),v_t)$, where $\mathbb{F}_p$ is the field with $p$
elements and $v_t$ is the canonical $t$-adic valuation on
$\mathbb{F}_p((t))$, we obtain a valued field extension $(L,v)$ of
$(\mathbb{F}_p((t)),v_t)$ with the following properties:
\par\noindent
a) $L|\mathbb{F}_p((t))$ is a regular extension of transcendence degree
1,
\par\noindent
b) $(L,v)$ is algebraically complete,
\par\noindent
c) the value group $vL$ is a $\mathbb{Z}$-group with smallest positive
element $v(t)$,
\par\noindent
d) the residue field $Lv$ is equal to $\mathbb{F}_p\,$,
\par\noindent
e) $v=w\circ v_t$ for the finest valuation $w$ that is coarser than $v$,
\par\noindent
f) the value group $wL=\mathbb{Q}$ is divisible and by Lemma~\ref{ww},
$(L,w)$ is algebraically complete,
\par\noindent
g) the residue field $Lw=\mathbb{F}_p((t))$ is a large field as it
carries a henselian valuation,
\par\noindent
h) for the polynomial $G(X_0,X_1)=X_0^p-X_0 +tX_1^p-x$ (where $x=s^{-1}$
in the notation of \cite{kuhlmann4}), $w(G(x_0,x_1))< 0$ for all
$x_0,x_1\in L$ (because it is shown in \cite{kuhlmann4} that if
$w(G(x_0,x_1))\geq 0$, then $x_1$ must be transcendental over $L$).

The equality $Lw=\mathbb{F}_p((t))$ holds because in the construction,
$(L,w)$ is an immediate extension of a field $(L_2,w)$ with residue
field $L_2w=K=\mathbb{F}_p((t))$.

Using the notation of \cite{kuhlmann4}, in particular the recursive
definition
\[\xi_1\>=\>x^{1/p} \quad\mbox{ and }\quad \xi_{j+1}\>=\>
(\xi_j-c_1s^{-p/q_j})^{1/p}\]
with $c_1=t$, and setting
\[a_k\>=\>\sum_{j=1}^{k}\xi_j\quad\mbox{ and }\quad
b_k\>=\>\sum_{j=1}^{k-1} s^{-1/q_j}\;,\]
we compute:
\begin{eqnarray*}
x-(a_k^p-a_k)-c_1 b_k^p & = & x-\sum_{j=1}^{k}(\xi_j^p-\xi_j)
-c_1 \sum_{j=1}^{k-1} s^{-p/q_j}\\
 & = & x-\xi_1^p -\sum_{j=1}^{k-1}(\xi_{j+1}^p-\xi_j) +\xi_k
-c_1 \sum_{j=1}^{k-1} s^{-p/q_j}\\
 & = & \sum_{j=1}^{k-1} c_1 s^{-p/q_j} +\xi_k
-\sum_{j=1}^{k-1} c_1 s^{-p/q_j}\\
 & = & \xi_k\;.
\end{eqnarray*}
Hence,
\[w(G(a_k,b_k))\>=\>w\xi_k\>=\>-\frac{1}{p^k}ws\><\>0\;.\]
This shows that $\{w(G(a_k,b_k))\mid k\in \mathbb{N}\}$ is cofinal in
the negative part $wL^{<0}$ of $wL$.

By the definition of $a_k$, we have $w(a_k)=w(\xi_1)=\frac{1}{p}w(x)=
-\frac{1}{p}w(s)<0$ for all $k$. Similarly, $w(b_k)=w(s^{-1/q_1})=
-\frac{1}{q_1}w(s)<0$ for all $k$. Consequently, $w(sa_k)>0$,
$w(sb_k)>0$ and hence $sa_k,sb_k\in \ca{O}_v\subset\ca{O}_w$. Now set
\[F(X,Y)\>=\>G(s^{-1}X,s^{-1}Y)=s^{-p}X^p-s^{-1}X
+ts^{-p}Y^p-s^{-1}\;.\]
Then $w(F(sa_k,sb_k))=w(G(a_k,b_k))$, but we still have that
$w(F(a,b))< 0$ for all $a,b\in L$. So we see that $\{w(F(a,b))\mid
a,b\in \ca{O}_w\}$ is a cofinal subset of $wL^{<0}$ and thus has no
maximal element. This proves that $(L,w)$ is not extremal.

Observe that $w(F(a,b))<0$ is equivalent to $v(F(a,b))<
v_t\mathbb{F}_p((t))$. As the set $\{w(F(a,b))\mid a,b\in \ca{O}_w\}$ is
a cofinal subset of $wL^{<0}$, the set $\{v(F(a,b))\mid a,b\in
\ca{O}_v\}$ is a cofinal subset of the set of all values
$<v_t\mathbb{F}_p((t))$ in $vL$ and thus has no maximal element. This
proves that $(L,v)$ is not extremal.

\end{section}

\begin{section}{Some further results}

An Ax-Kochen-Ershov Principle as in Theorem~\ref{AKE} also holds for
formally $\wp$-adic fields (see \cite{preroq}) and, more generally, for
finitely ramified fields (see \cite{ershov}, \cite{ziegler}). A formally
$\wp$-adic field is $\wp$-adically closed if and only if it is henselian
and its value group is a $\mathbb{Z}$-group. Formally $\wp$-adic and
finitely ramified fields are algebraically complete as soon as they are
henselian. Hence, we obtain from Theorems~\ref{zgroup} and~\ref{char0}
via the Ax-Kochen-Ershov Principle:

\begin{theorem}                                  \label{fr}
A formally $\wp$-adic field is extremal if and only if it is
$\wp$-adically closed. A finitely ramified field is extremal if and
only if it is henselian and its value group is a $\mathbb{Z}$-group.
\end{theorem}

\par\smallskip
If $\ca{K}=(K,\Gamma,\bk;v)$ is a valued field such that $v=w\circ\bar
w$, then by Lemma~\ref{ww}, $v$ is henselian if and only if $w$ and
$\bar w$ are. The same holds for ``algebraically complete'' in the place
of ``henselian''. The corresponding assertion for ``extremal'' is not
entirely known. We leave the easy proof of the following result to the
reader:
\begin{lemma}
If $\ca{K}$ is extremal, then $K$ is also
extremal with respect to every coarsening $w$ of $v$.
\end{lemma}

From our characterization of extremal fields we obtain:

\begin{proposition}
Let $\ca{K}=(K,\Gamma,\bk;v)$ be a perfect algebraically complete valued
field such that $\cha{\bk}=\cha{K}$ and $v=w\circ\bar{w}$. Then $\ca{K}=
(K,\Gamma,\bk;v)$ is extremal if and only if $K$ is extremal with
respect to $w$ and the residue field of $K$ under $w$ is extremal
with respect to $\bar{w}$.
\end{proposition}

\end{section}

\bibliographystyle{plain}

\end{document}